\documentclass[12pt,reqno]{amsart}
\usepackage[utf8]{inputenc}
\usepackage{ytableau}
\usepackage{youngtab}
\usepackage{comment}
\usepackage{mathtools}
\usepackage{tikz-cd}
\usepackage{color}
\usepackage{hyperref}
\usepackage{amsmath, latexsym}
\usepackage{amssymb}
\usepackage{amsfonts}
\usepackage{graphicx}
\usepackage[all]{xy}
\usepackage{verbatim}
\usepackage{mathdots}
\usepackage{cleveref}
\usepackage[left=2.8cm,right=3.0cm,top=3.0cm,bottom=3.2cm]{geometry}
\usepackage[colorinlistoftodos]{todonotes}
\hypersetup{colorlinks,linkcolor={blue},citecolor={blue},urlcolor={blue}}

\newtheorem{theorem}{Theorem}[section]
\newtheorem{lemma}[theorem]{Lemma}
\newtheorem{corollary}[theorem]{Corollary}
\newtheorem{proposition}[theorem]{Proposition}

\newtheorem*{theorem-nonumber}{Theorem}
\newtheorem*{definition-nonumber}{Definition}
\newtheorem*{proposition-nonumber}{Proposition}

\theoremstyle{definition}
\newtheorem{definition}[theorem]{Definition}
\newtheorem{example}[theorem]{Example}
\newtheorem{remark}[theorem]{Remark}
\newtheorem{examplemat}{Example}

\newcommand{\K}{\mathcal{K}}

\newcommand{\X}{\textsc{X}}

\newcommand{\Jm}{\textsc{J}}
\newcommand{\grass}{\textsc{Sp}\textsc{G}\mathsf{r}}
\newcommand{\Gr}{\textsc{G}\mathsf{r}}
\newcommand{\Sp}{\textsc{Sp}}

\newcommand{\U}{\textsc{U}}
\newcommand{\I}{\textsc{I}}

\newcommand{\T}{\textsc{T}}

\newcommand{\complete}{\textsc{SpF}}
\newcommand{\ov}{\overline}

\newcommand{\ldb}{\{\!\!\{}
\newcommand{\rdb}{\}\!\!\}}

\newcommand{\BB}{\mathcal{B}}
\newcommand{\PP}{\mathbb{P}}
\newcommand{\TP}{\mathbb{T}\mathbb{P}}
\newcommand{\RR}{\mathbb{R}}
\newcommand{\TT}{\mathbb{T}}
\newcommand{\CC}{\mathbb{C}}

\newcommand{\Dr}{\textsc{D}\mathsf{r}}
\newcommand{\trop}{\operatorname{trop}}
\newcommand{\val}{\operatorname{val}}
\newcommand{\cone}{\operatorname{cone}}
\newcommand{\Bergman}{\operatorname{Be}}
\newcommand{\KK}{\mathbb{K}}

\newcommand{\sign}{\operatorname{sign}}

\newcommand{\tdet}{\operatorname{tdet}}
\newcommand{\iden}{\operatorname{Id}}
\newcommand{\conv}{\operatorname{conv}}

\begin{document}
\nocite{*}
\thispagestyle{empty}
\title[The tropical symplectic Grassmannian]{The tropical symplectic Grassmannian}
\author[G. Balla and J. A. Olarte]{George Balla and Jorge Alberto Olarte}
\address{Algebra and Representation Theory, RWTH Aachen University, Pontdriesch 10-16, 52062 Aachen, Germany}
\email{balla@art.rwth-aachen.de}
\address{Institut f\"ur Mathematik, Technische Universit\"at Berlin, Stra{\ss}e des 17. Juni 135, 10623 Berlin, Germany}
\email{olarte@math.tu-berlin.de}
\maketitle

\begin{abstract}
    We launch the study of the tropicalization of the symplectic Grassmannian, that is, the space of all linear subspaces isotropic with respect to a fixed symplectic form.
    We formulate tropical analogues of several equivalent characterizations of the symplectic Grassmannian and determine all implications between them. In the process, we show that the Pl\"ucker and symplectic relations form a tropical basis if and only if the rank is at most 2. 
    We provide plenty of examples that show that several features of the symplectic Grassmannian do not hold after tropicalizing. We show exactly when do conormal fans of matroids satisfy these characterizations, as well as doing the same for a valuated generalization. Finally, we propose several directions to extend the study of the tropical symplectic Grassmannian. 
\end{abstract}

\section{Introduction}
Given a field $\KK$, the Grassmannian, denoted by $\Gr_{\KK}(k,n)$, is the space of all linear subspaces of $\KK^n$ of dimension $k$. Let $\omega$ denote the standard symplectic form on $\KK^{2n}$. The symplectic Grassmannian, which will be denoted by $\grass_{\KK}(k,2n)$, is the subset of the Grassmannian consisting of all isotropic subspaces of $\KK^{2n}$, that is, spaces whose elements are orthogonal to each other with respect to $\omega$.
We consider the Grassmannian under its Pl\"ucker embedding and the symplectic Grassmannian as a subset under the same embedding. The defining ideal of the latter variety with respect to this embedding is generated by Pl\"ucker and certain linear relations that we term symplectic relations \cite{Dec79}. We will therefore refer to this ideal as the Pl\"ucker-Symplectic ideal.
For a given valuation $\val: \KK \to \TT := \RR\cup\{\infty\}$, the tropicalization of the Grassmannian is the space of tropicalizations of linear subspaces \cite{SS04}. We define the tropical symplectic Grassmannian as the tropicalization of the symplectic Grassmannian with respect to the Pl\"ucker-Symplectic ideal. Therefore, it is a subset of the tropical Grassmannian consisting of tropicalizations of isotropic linear subspaces. We will denote it by $\T\grass_p(k,2n)$, where $p$ is the characteristic of the field $\KK$.\\

It is known that the geometry of the tropical Grassmannian is governed by the combinatorics of matroids \cite{Spe08}. Coxeter matroids are generalizations of matroids to different root systems, where the type {\tt A} Coxeter matroids are the usual matroids \cite{BGW03}. In \cite{Fel12}, Rinc\'on studied the tropicalization of the type {\tt D} Grassmannian, known as the Spinor variety, using the type {\tt D} Coxeter matroids, also known as $\Delta$-matroids, to obtain similar results as those in \cite{Spe08} for type {\tt A}. The symplectic Grassmannian can be interpreted as the type {\tt C} Grassmannian. However, very little is known about the type {\tt C} Coxeter matroids, known as symplectic matroids (see  \cite{BGW98} and \cite[Chapter 3]{BGW03} for what is known about them). In particular, there is no good characterization of symplectic matroids in terms of axioms for their bases. This makes it more challenging to build a theory of their valuated counterparts as tropical Pl\"ucker vectors similar to the ones for types {\tt A} and {\tt D}. So, we take a different approach, relying on the already rich theory of type {\tt A} valuated matroids.\\

There are several equivalent ways of saying that a linear space belongs to the symplectic Grassmannian, among them:
\begin{enumerate}
    \item Its Pl\"ucker vector satisfies all polynomials in the Pl\"ucker-Symplectic ideal.
    \item Its Pl\"ucker vector satisfies the symplectic relations.
    \item It is isotropic.
    \item It has a basis of pairwise orthogonal vectors (with respect to $\omega$).
    \item It is the row span of a matrix $(A|B)$ such that $A \cdot B^T$ is symmetric, where $A,B\in \KK^{k\times n}$.
\end{enumerate}

There is a natural way of obtaining a tropical analogue for each of the statements above, by imposing a condition on a given tropical linear space. The first is the most obvious one, as it just asks to be in the tropicalization of the symplectic Grassmannian. Just like the Dressian is the tropical prevariety consisting of vectors satisfying the tropical Pl\"ucker relations, we define the symplectic Dressian, which will be denoted by $\Sp\Dr(k,2n)$, to be the prevariety consisting of vectors satisfying the tropical symplectic relations as well as the tropical Pl\"ucker relations. 

A tropical notion of isotropic linear spaces was already studied in \cite{Fel12}. Rinc\'on's motivation was in type {\tt D}, that is, in the context of isotropic spaces with respect to a symmetric bilinear relation. However, the sign that distinguishes that context to ours vanishes after tropicalizing, so his definition is the relevant one for us too. For the last two characterizations, we use the tropical Stiefel map from \cite{FR15} as an analogue of being the span of a collection of vectors. However, we recognize the limitations of this last analogy, since it forces us to restrict ourselves to tropical linear spaces of this kind. 
Thus we obtain the following conditions on a given tropical linear space that can be considered tropical analogues of the characterizations of the symplectic Grassmannian above.
\begin{enumerate}
    \item It is in the tropical symplectic Grassmannian.
    \item It is in the symplectic Dressian.
    \item It is tropically isotropic.
    \item It is the image under the tropical Stiefel map of a matrix with (tropically) orthogonal rows.
    \item It is the image under the tropical Stiefel map of a matrix $(A|B)$ such that $A\odot B^T$ is symmetric, where $A,B\in \TT^{k\times n}$.
\end{enumerate}

Our main result (\Cref{thm:main}), determines for each $k$ (dimension of the tropical linear space) and $n$ (dimension of the ambient space), which of the above statements implies another. It is fully described in \Cref{fig:zoomap} below. 
\begin{figure}[ht]
	\centering
		\includegraphics[width =0.98\textwidth]{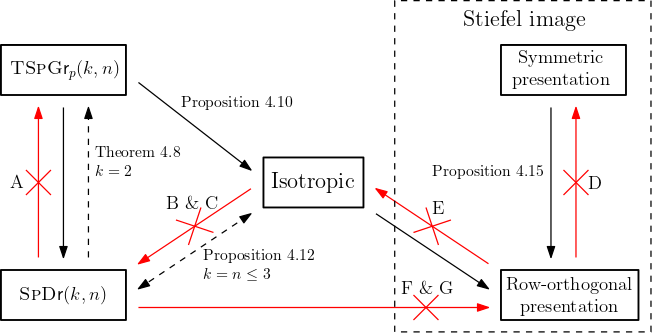}
	\caption{Map of the Matroid Zoo.}
		\label{fig:zoomap}
\end{figure}
The five boxes represent the five tropical analogues. Black full arrows show implications, which are containment of sets. For example, the left most downward arrow means that $\T\grass_p(k,2n)\subseteq \Sp\Dr(k,2n)$. Dashed arrows represent implications for specific cases. They come with a label indicating the conditions on $k$ and $n$ for them to hold. Black arrows come labeled with the corresponding statement in \Cref{sec:tropicalanalogies} when not trivial. 
Red arrows represent counterexamples that show that such implications fail to hold. They come with the letters that label the corresponding counterexample in \Cref{sec:matroidzoo}.
The dashed box titled `Stiefel image' is there to clarify that the arrows which interact with the two right boxes are taken under the assumption that we are restricted to the Stiefel image. It is not there to say that only these valuated matroids are in the Stiefel image; for example, we have that the matroid of \Cref{ex:4points} is also in the Stiefel image.\\

A particularly interesting case covered in our main theorem is the rank two case. We see that in this case, the tropical symplectic Grassmannian $\T\grass_p(2,2n)$ and the symplectic Dressian $\Sp\Dr(2,2n)$ coincide (\Cref{thm:rank2}). This result also implies that the Pl\"ucker and symplectic relations form a tropical basis for the Pl\"ucker-Symplectic ideal in this case. This is analogous to the type {\tt A} result by Maclagan and Sturmfels \cite{MS15}. Furthermore, Speyer and Sturmfels \cite{SS04} showed that the tropical Grassmannian in rank two coincides with the space of phylogenetic trees. We show that this coincidence holds in the symplectic case as well, after quotienting $\T\grass_p(2,2n)$ by the tropical hyperplane corresponding to the single tropical symplectic relation (Proposition \ref{pro:coincidence}). We then go on to obtain enumerative information on this fan, namely, the number of rays and facets (Corollary \ref{cor:rays}) and the Betti numbers for the corresponding simplicial complex (\Cref{cor:betti}).\\

Another case of particular interest is the Lagrangian case, that is, when $k=n$. In \cite{ADH20}, the authors proved that the $h$-vector of the independence complex and the broken circuit complex of a matroid is log-concave, solving  conjectures that were standing since the 80's. One of their main tools is what they call the conormal fan of a matroid $M$, which has the same support as the product of the Bergman fan of $M$ and its dual. The conormal fan has a natural generalization to any valuated matroid, forming a class of linear spaces which make sense to study under our context. We show that indeed, they are tropically isotropic, they satisfy the tropical symplectic relations, and provided the involved valuated matroids are realizable, they are in the tropical symplectic Grassmannian (\Cref{prop:conormal}).\\

The paper is organized as follows: \Cref{sec:symplecticgrassmannian} provides background on the classical Grassmannian as well as the symplectic Grassmannian while \Cref{sec:tropicalgrassmannian} provides background on the tropical Grassmannian and its connection to valuated matroids. In \Cref{sec:tropicalanalogies}, we formulate each of the tropical analogues of characterizations of the symplectic Grassmannian and show the implications between them. \Cref{sec:ranktwo} deals with the rank two case and in particular, we show that for this case, the symplectic and Pl\"ucker relations form a tropical basis for the symplectic Grassmannian. \Cref{sec:matroidzoo} contains the other side of the main result, by exhibiting all counterexamples needed to finish the proof, plus some other interesting examples that exhibit some pathologies that can occur. Before that, \Cref{sec:directsums} explains why the list given in \Cref{sec:matroidzoo} is complete to prove \Cref{thm:main}. 
In \Cref{sec:conormal}, we answer the question: where do conormal bundles fit in \Cref{fig:zoomap}? 
Finally, \Cref{sec:future} suggests several directions for future work, including the connection to symplectic matroids and applications to toric degenerations of flag varieties. 

\subsection*{Acknowledgments}
The authors are grateful to Ghislain Fourier for suggesting this project and for insightful discussions. The second author would also like to thank Felipe Rinc\'on for useful discussions. Great thanks to Xin Fang, Christian Steinert and an anonymous referee for reading an earlier version, and for pointing out misprints. The first author is funded by the Deutscher Akademischer Austauschdienst (DAAD, German Academic Exchange Service) scholarship program: Research Grants - Doctoral Programs in Germany. The second author is funded by the Deutsche Forschungsgemeinschaft (DFG, German Research Foundation) through Project-ID 286237555 – TRR 195 and under Germany's Excellence Strategy; The Berlin Mathematics Research Center MATH+ (EXC-2046/1, project ID 390685689, 
sub-project AA/EFx-y).

\numberwithin{equation}{section}
\section{The Symplectic Grassmannian}
\label{sec:symplecticgrassmannian}
We begin by fixing some notation that we will be using. Let $[n]:=\{1,\ldots,n\}$ and $[2n]:=\{1,\ldots,n,\overline{1},\ldots, \overline{n}\}$ ordered as  $1<\cdots<n<\overline{1}<\cdots< \overline{n}$. Adding a bar is a dual operation, so that $\bar{\bar{i}} = i$.
For a set $S$, we write $Si := S\cup\{i\}$ and $\overline{S} := \{\bar{i} \mid i\in S \}$.
Let $\KK$ be an algebraically closed field of characteristic $p$. 
For a matrix $M\in \KK^{k\times n}$ we write $M^T$ for the transpose and for a sub-sequence $\Jm\subseteq [n]$ let $M_{\Jm}$ denote the sub-matrix of $M$ with columns indexed by $\Jm$. 

The \emph{Grassmannian} $\Gr_\KK(k,n)$ is the space of all $k$-dimensional linear subspaces of $\KK^{n}$. 
Whenever the field $\KK$ is not relevant, we omit it from the subscript and write just $\Gr(k,n)$.
A subspace $L\subseteq\KK^{n}$ can be identified by its \emph{Pl\"ucker coordinates} $\X_{\Jm}$, which are determinants of the sub-matrices $M_{\Jm}$ for every $\Jm\in \binom{[n]}{k}$, where $M$ is a $(k\times n)$-matrix whose rows span $L$. A \textit{Pl\"ucker vector} $[\X]\in\PP^{\binom{n}{k}-1}$ is a vector whose coordinates are the $\binom{n}{k}$ Pl\"ucker coordinates. The Grassmannian $\Gr(k,n)$ can be identified with the set of all Pl\"ucker vectors in $\PP^{\binom{n}{k}-1}$ via the \emph{Pl\"ucker embedding}. It is a projective variety  of dimension $k(n-k)$ that satisfies the \emph{Pl\"ucker relations}:
\begin{equation}\label{eqn:plueckerrelations}
    \forall \,\,S\in \binom{[n]}{k-1}, \quad \forall \,\,T\in \binom{[n]}{k+1}, \quad\sum_{i\in T\backslash S} \sign(i,T\backslash S)\X_{T\setminus i}\cdot \X_{Si} = 0,
\end{equation}
where the sign $\sign(i,T\backslash S)\in\{-1,1\}$ alternates with respect to the ordering of $T\backslash S$. We write $\mathcal{P}_{k,n}(\KK)$ for the ideal generated by the above relations, called the \emph{Pl\"ucker ideal}.\\

The subspace $L\subseteq\KK^{n}$ can be recovered from the Pl\"ucker vector $\X\in \PP^{\binom{n}{k} -1}$ as follows:
\begin{equation}
\label{eq:linspace}
L = \left\{x\in \KK^{n} \mid \forall\,\, T\in \binom{[n]}{k+1},\quad \sum_{i\in T} \sign(i,T)\X_{T\setminus i}\cdot x_i = 0\right\}.
\end{equation}

A symplectic vector space $W$ is a vector space with a symplectic form, which is a bilinear form $\omega: W\times W \to \KK$ satisfying:
\begin{itemize}
    \item $\forall\,\, v\in W \enspace \omega(v,v) = 0 $ (alternating).
    \item if $\omega(u,v) = 0$ for every $u\in W$, then $v=0$ (non-degenerate).
\end{itemize}
The first condition implies that $\omega(u,v) = - \omega(v,u) \,\, \forall\,\, u,v\in W$, however the converse only holds if $p\ne 2$.

It is well known that any symplectic vector space has even dimension and that there is a basis $\{e_1,\dots, e_n,e_{\bar{1}},\dots, e_{\bar{n}}\}$ such that $\omega$ is given by $\omega(e_i, e_{\bar{i}}) = 1$ (for $i<\bar{i}$) and $\omega(e_i,e_j)= 0$ for $j\ne \bar{i}$. In other words, $\omega$ can be written as
\begin{equation}
\label{eq:sympform}
\begin{pmatrix}
0 & \iden_n \\
-\iden_n & 0
\end{pmatrix},    
\end{equation}
with respect to this basis, where $\iden_n$ is the $n\times n$ matrix with $1$'s on the diagonal and zeros elsewhere.

We say that two vectors $u$ and $v$ are \emph{orthogonal} if $\omega(u,v) = 0$. A $k$-dimensional linear subspace $L\subseteq \KK^{2n}$ is called \emph{isotropic} if every two vectors in $L$ are orthogonal. 
By bilinearity, it is enough that $L$ has a basis consisting of pairwise orthogonal vectors. If we write such a basis as rows of a matrix $(A| B)$ with $A$ and $B$ two $k\times n$ matrices,
it is straightforward to see that the basis is orthogonal if and only if $A\cdot B^T$ is symmetric. If $L$ is isotropic then it follows that $k\le n$. When $\dim(L) = n$, we also call $L$ \emph{Lagrangian}.\\

The \emph{symplectic Grassmannian} $\grass_{\KK}(k,2n)$ is the sub-variety of $\Gr_{\KK}(k,2n)$ consisting of all isotropic linear subspaces of $\KK^{2n}$. Again, we omit the subscript $\KK$ when the field does not play a role and write $\grass(k,2n)$. It is an irreducible projective variety of dimension $\frac{k(4n-3k+1)}{2}$ (see for example, \cite{MWZ20}).
It follows from the work of De Concini \cite{Dec79} that the defining ideal of $\grass(k,2n)$, which we denote by $S_{k,2n}(\KK)$, is generated by the Pl\"ucker relations together with the \emph{symplectic relations}:
\begin{equation}\label{eqn:symplecticrelations}
    \forall\,\, S\in \binom{[2n]}{k-2},  \quad \sum_i \X_{Si\overline{i}} = 0,
\end{equation}
with respect to the Pl\"ucker embedding, where the sum goes over all $i\in[n]$ such that $\{i,\overline{i}\}\cap S =\emptyset$. 
We call $S_{k,2n}(\KK)$ the \emph{Pl\"ucker-Symplectic ideal}. 
We include the following short proof that the symplectic relations really define the symplectic Grassmannian, for completeness. It was originally written by Carrillo-Pacheco and Zaldivar \cite{CPZ11} for the Lagrangian case $k=n$, but it works for any $k$. 

\begin{proposition}
The defining ideal of $\grass(k,2n)$ with respect to the Pl\"ucker embedding is the ideal $S_{k,2n}(\KK)$.
\end{proposition}
\begin{proof}
For $1\leq k\leq n$, let $\bigwedge^k\KK^{2n}$ be the $k$-th exterior power over $\KK^{2n}$. Define a linear map:
\begin{align*}
    f: \bigwedge^k \KK^{2n} &\longrightarrow \bigwedge^{k-2} \KK^{2n},\\
    v_1\wedge\cdots \wedge v_k &\longmapsto \sum_{1\leq r<s\leq k}(-1)^{r+s}\omega(v_r,v_s)v_1\wedge \cdots\wedge \hat{v}_r\wedge\cdots\wedge \hat{v}_s\wedge\cdots\wedge v_k,
\end{align*}
where $\hat{v}$ means that the corresponding term is omitted. 
Notice that the sign $(-1)^{r+s}$ is necessary for $f$ to be well defined with respect to commuting the terms in the wedge product (this sign is missing in \cite{CPZ11}).
Let $\K$ denote the kernel of $f$ and denote by $\PP(\K)$ the projectivization of $\K$. Then $\PP(\K)$ is a closed irreducible subset of $\PP\left( \bigwedge^k \KK^{2n}  \right)$ under the Pl\"ucker embedding. 

Claim 1: $\grass(k,2n)=\Gr(k,2n)\cap \PP(\K)$. It is clear that $\grass(k,2n)\subseteq\Gr(k,2n)\cap \PP(\K)$. For the opposite inclusion, if $v\in \Gr(k,2n)\cap \PP(\K)$, then $v$ is the equivalence class of $v_1\wedge\cdots\wedge v_k$ with $v_1,\ldots, v_k$ linearly independent and $f(v)=0$, that is:
\[f(v)=\sum_{1\leq r<s\leq k}(-1)^{r+s}\omega(v_r,v_s)v_1\wedge \cdots\wedge \hat{v}_r\wedge\cdots\wedge \hat{v}_s\wedge\cdots\wedge v_k=0,\]
and thus $\omega(v_r,v_s)=0$ for all $1\leq r<s\leq k$, and hence the claim follows. 

Claim 2: $v\in \K$ if and only if the linear relations \eqref{eqn:symplecticrelations} vanish. 
For this, we use the linearity of $f$. Write $v$ with its Pl\"ucker coordinates, that is, 
\[v=\sum\limits_{I\in \binom{[2n]}{k}} \X_I e_{i_1}\wedge\dots\wedge e_{i_k}\] where $I=\{i_1\dots i_k\}$. So 
\[
f(v) = \sum\limits_{I\in \binom{[2n]}{k}} \X_I f(e_{i_1}\wedge\dots\wedge e_{i_k}).
\]
Notice that $f(e_{i_1}\wedge\dots\wedge e_{i_k})$ vanishes, unless $\{i,\bar{i}\}\subseteq I$ for some $i$. Let $S= \{s_1\dots s_{k-2}\}\in  \binom{[2n]}{k-2}$. The coefficient of $e_{s_1}\wedge \dots \wedge e_{s_k}$ in $f(v)$ is
\[
\sum_{\{i,\bar{i}\}\cap S=\emptyset} \X_{Si\bar{i}}(-1)^{2k-1}.
\]
So $f(v)=0$ if and only if the coefficient above is $0$ for every $S$, which are precisely the linear relations (\ref{eqn:symplecticrelations}).
\end{proof}

To summarize, we have established five equivalent conditions for a linear space $L$ to be in the symplectic Grassmannian $\grass(k,2n)$:
\begin{enumerate}
    \item All polynomials in the Pl\"ucker-Symplectic ideal $S_{k,2n}(\KK)$ vanish on the Pl\"ucker coordinates of $L$. 
    \item The Pl\"ucker coordinates of $L$ satisfy the symplectic relations (\Cref{eqn:symplecticrelations}).
	\item For every pair $u,v \in L$, we have $\omega(u,v) = 0$.
	\item There is a basis of $L$,  $u_1,\dots,u_k$, such that $\forall \,\,i,j, \enspace \omega(u_i,u_j) = 0$.
	\item If $L$ is the row-span of a matrix $(A | B)$, then $A\cdot B^T$ is symmetric, where $A,B\in \KK^{k\times n}$.
\end{enumerate}

\section{The Tropical Grassmannian}\label{sec:tropicalgrassmannian}
In this section, we recall notions related to the tropical Grassmannian, which was first introduced and studied by Speyer and Sturmfels \cite{SS04}. We refer the reader to Maclagan and Sturmfels \cite{MS15} for a comprehensive introduction to tropical geometry. 

To get started, we recall a couple of key tropical notions and fix some conventions. Let $\TT$ be the semi-field of tropical numbers $\TT = \RR\cup\{\infty\}$ with operations $\oplus := \min$ and $\odot := +$.
Let $\val:\KK\to \TT$ be a non-Archimidean valuation. 
Given an ideal $\I\subseteq\KK[\X_1,\ldots,\X_n]$, let $V(\I)$ denote the algebraic variety corresponding to the ideal $\I$. The \textit{tropical variety} $\trop(V(\I))$ is the intersection of the tropical hypersurfaces $\trop(V(f))$ for \textbf{all} $f\in \I$. 

On the other hand, a \textit{tropical prevariety} is the intersection of finitely many tropical hypersurfaces, which is not necessarily a tropical variety. 
A finite set of polynomials is said to be a \textit{tropical basis} of the ideal it generates if the tropical prevariety it defines is equal to the tropical variety of that ideal.

\subsection{The Tropical Grassmannian and the Dressian}
First, we recall the \emph{tropical projective space}
\[
\TT\PP^{\binom{n}{k}-1} := \left(\TT^{\binom{n}{k}} \setminus (\infty,\dots,\infty)\right)/ \RR(1,\dots,1).
\]
The \emph{tropical Grassmannian}, which we denote by $\T\Gr_p(k,n)$, is the tropicalization of $\Gr(k,n)$ i.e., $\T\Gr_p(k,n) = \trop(\Gr(k,n))$. 
We want to emphasize that we use the notation $\T\Gr_p(k,n)$ because the tropicalization of the Pl\"ucker ideal only depends on the characteristic $p$ of $\KK$. However, similarly as before, whenever we say something about the tropical Grassmannian that holds for every characteristic, we omit $p$ from the subscript and write $\T\Gr(k,n)$.

The \emph{Dressian} $\Dr(k,n)\supseteq \T\Gr(k,n)$ is the prevariety given by all the Pl\"ucker relations, i.e., all vectors in $\TP^{\binom{n}{k}-1}$ that satisfy the \emph{tropical Pl\"ucker relations}:
\[
\forall\,\, S\in \binom{[n]}{k-1}, \enspace \,\,\forall\,\, T\in \binom{[n]}{k+1}, \,\, \text{the minimum in} \bigoplus_{i\in T\setminus S} (x_{T\setminus i} \odot  x_{Si})
\]
is achieved at least twice (or is equal to $\infty$).

A vector $\mu\in \Dr(k,n)$ is called a \emph{valuated matroid} 
or a \emph{tropical Pl\"ucker vector}. A \emph{matroid} can be defined to be a tropical Pl\"ucker vector consisting of just $0$ and $\infty$ as entries. The set of subsets whose corresponding coordinate is $0$ is the usual set of \emph{bases} of the matroid (see Oxley \cite{OXL06} for more on the theory of ordinary matroids).
 
 The difference between the tropical Grassmannian and the Dressian is an extension of representable matroids vs non-representable matroids. When the bases of a matroid are the set of non-zero coordinates of a point in the Grassmannian $\Gr_\KK(k,n)$, we say that the matroid is \emph{representable} or \emph{realizable} over the field $\KK$. The existence of non-realizable matroids for fields of a given characteristic, such as the Fano matroid (for characteristic different than 2) and non-Fano matroid (for characteristic equal to 2), implies that the tropical Grassmannian indeed depends on the characteristic and that the Pl\"ucker relations \eqref{eqn:plueckerrelations} do not form a tropical basis for the Pl\"ucker ideal $\mathcal{P}_{k,n}$.
 In other words, the tropical Grassmannian $\T\Gr_p(k,n)$ and the Dressian $\Dr(k,n)$ do not coincide in general, except in some cases:

\begin{theorem}[Maclagan-Sturmfels, \cite{MS15}]
\label{thm:rank2A}
The Dressian $\Dr(2,n)$ and the tropical Grassmannian $\T\Gr(2,n)$ coincide.
\end{theorem}

In an analogous fashion as before, a \emph{tropical linear space} can be recovered from a tropical Pl\"ucker vector $\mu$:

\begin{equation}
\label{eq:troplin}
L_{\mu} = \left\{x\in \TT^n \mid \forall \,\,T\in \binom{[n]}{k+1}, \text{ the minimum in}\,\, \bigoplus\limits_{i\in T} \mu_{T\setminus i}\odot x_i\,\, \text{is A.A.T or equals } \infty\right\},  
\end{equation}
where A.A.T means ``achieved at least twice''. If $\mu\in \T\Gr_p(k,n)$ has rational coefficients and $\KK$ is algebraically closed, then by the Fundamental Theorem of Tropical Algebraic Geometry \cite[Theorem 3.2.3]{MS15}, $\mu = \val(\X)$ for some Pl\"ucker vector $\X$ over $\KK$.
 
The tropical linear space corresponding to $\mu$ coincides with the tropicalization of the linear space corresponding to $\X$, i.e., $\trop(L_{\X}) = L_{\mu}$ \cite[Theorem 3.8]{SS04}.
If $\mu$ is a valuated matroid which is not in the tropical Grassmannian, $L_{\mu}$ is not the tropicalization of a linear space, however there exist reasons to still call such $L_{\mu}$ a tropical linear space (see, for example, \cite[Theorem 6.5]{FIN13}).\\

The study of tropical linear spaces equals the study of the valuated matroids, which relies heavily on a polyhedral view of matroids. Valuated matroids are characterized by being vectors which induce a regular subdivision of a matroid polytope into matroid polytopes (Speyer \cite{Spe08}). The tropical linear space of an unvaluated matroid (such as the case where the valuation is trivial) is a fan whose support is the same as a fan known as the \emph{Bergman fan}.

The condition that a linear subspace contains another can be read off from their Pl\"ucker vectors via the so called incidence relations. 
In \cite{HAQ12}, Haque shows that the incidence relations tropicalize well:

\begin{proposition}[\cite{HAQ12}]
\label{prop:incidence}
Let $\mu_1$ and $\mu_2$ be two valuated matroids on $[n]$ of rank $k_1 \le k_2$. Then $L_{\mu_1} \subseteq L_{\mu_2}$ if and only if 
\[
\forall S\in\binom{[n]}{k_1-1}, \enspace T\in \binom{[n]}{k_2+1} \text{ the minimum in } \bigoplus_{i\in T\setminus S} (\mu_1)_{Si}\odot (\mu_2)_{T\backslash i}
\]
is achieved at least twice.
\end{proposition}

\subsection{The Tropical Stiefel Map}

The classical Stiefel map sends full rank $k\times n$ matrices to their row span. 
In \cite{FR15}, Fink and Rinc\'on introduced a tropical analogue. Let $A\in \TT^{k\times n}$ be a matrix with tropical entries. For $\Jm\in \binom{[n]}{k}$, let $A_{\Jm}$ be the sub-matrix with columns indexed by $\Jm$. Then:

\begin{definition}[\cite{FR15}]\label{def:stiefel}
The \emph{tropical Stiefel map} $\pi: \TT^{k\times n}  \dashrightarrow  \TT^{\binom{[n]}{k}}$ is given by
\[
\pi(A)_{\Jm} := \tdet(A_{\Jm}),
\]
where $\tdet$ denotes the tropical determinant, that is, the usual determinant computed with tropical operations. It is only defined for matrices such that at least one of such determinants is different from $\infty$, which is one notion of tropical full rank \cite{DSS05}.
\end{definition}

Not all valuated matroids arise this way. The matroids that arise this way are known as \emph{transversal matroids} and the corresponding valuated matroids are a natural generalization so they are also called \emph{transversal}. For a general introduction to (unvaluated) transversal matroids, we recommend \cite{BRU87}.

For some valuated matroid $\mu$, a matrix $A$ such that $\pi(A) = \mu$ is called a \emph{presentation} of $\mu$. We naturally have $\pi(A)\in \T\Gr_0(k,n)$ but not all valuated matroids that are representable over characteristic $0$ are of this form (see, for example, \cite[Figure 1]{FR15}). The space of transversal valuated matroids in $\T\Gr_0(k,n)$ is called the \emph{Stiefel image}.

The fibers of the tropical Stiefel map were explicitly described by Fink and the second author \cite[Theorem 6.6]{FO19}. This generalizes the characterization of presentations of valuated matroids by Brualdi and Dinolt \cite[Theorem 4.7]{BD72}. It is easy to see that any row of $A$ is a point in $L_{\pi(A)}$. In general, the fiber $\pi^{-1}(\mu)$ is the orbit of a certain fan in $L_{\mu}^k$ under the action of permuting rows.

\section{Tropical Analogies}\label{sec:tropicalanalogies}
In  this section, we describe the tropical analogues of the five equivalent classical realizations of isotropic subspaces discussed in Section \ref{sec:symplecticgrassmannian} following section \ref{sec:tropicalgrassmannian}, and then we go on to discuss the analogies between them. The key point to catch is that these five realizations are no longer tropically equivalent in general. 

\subsection{Tropical Analogues of the Five Equivalent Characterizations of Classical Isotropic Subspaces}
\label{subsec:anal}
We start by introducing the analogue of the tropical Grassmannian; the \emph{tropical symplectic Grassmannian} which will be denoted by $\T\grass(k,2n)$. 

\begin{definition}
The \emph{tropical symplectic Grassmannian} $\T\grass_p(k,2n)$ is the tropicalization of the symplectic Grassmannian $\grass_\KK$. That is
\[\T\grass_p(k,2n)=
\{w \in \TP^{\binom{2n}{k}-1} \mid w \in \trop(V(f)) \quad \forall \,\, f \in S_{k,2n}(\KK)\}.
\]
\end{definition}

Again, the tropical symplectic Grassmannian depends on the characteristic $p$ of the field $\KK$, but we omit the subscript when it is not relevant.

The following is a consequence of the Structure Theorem \cite[Theorem 3.3.5]{MS15}:
\begin{corollary}
The tropical symplectic Grassmannian $\T\grass(k,2n)$ is a  pure polyhedral fan in $\TP^{\binom{2n}{k}-1}$ of dimension $\frac{k(4n-3k+1)}{2}$.
\end{corollary}

Below we provide explicit examples of the fan $\T\grass(k,2n)$ for small $k$ and $n$ that we compute using the \emph{Tropical} package (Améndola et al \cite{TROP17}) based on \emph{Gfan} (Jensen \cite{JEN08}) in \emph{Macaulay2} (Grayson and Stillman \cite{GS97}). [We use the same software for all other computational examples in this paper.]

\begin{example}[k=2, n=3] The ideal $S_{2,6}(\KK)$ of $\grass(2,6)$ is generated by the corresponding Pl\"ucker relations and the symplectic relation: $\X_{1,\ov{1}} + \X_{2,\ov{2}} +\X_{3,\ov{3}}$. Here $\T\grass(2,6)$ is a pure $7$-dimensional fan in $\TP^{14}$ with a lineality space of dimension 3. Its f vector modulo the lineality space is $(28, 180, 420, 315)$, so it has 28 rays and 315 maximal cones.
\end{example}

\begin{example}[k=3, n=3] For $\grass_0(3,6)$, the ideal $S_{3,6}(\CC)$ is generated by the respective Pl\"ucker relations and the following set of symplectic relations:
\[\small\X_{2,\ov{1},\ov{2}} + \X_{3,\ov{1},\ov{3}},\,\, \X_{1,3,\ov{1}} + \X_{2,3,\ov{2}},\,\, -\X_{1,\ov{1},\ov{2}}+\X_{3,\ov{2},\ov{3}},\]
\[\X_{1,\ov{1},\ov{3}} +\X_{2,\ov{2},\ov{3}},\,\, \X_{1,2,\ov{1}} + \X_{2,3,\ov{3}},\,\, \X_{1,2,\ov{1}} - \X_{2,3,\ov{3}}.\]
The fan $\T\grass_0(3,6)$ is pure of dimension 6 in $\TP^{19}$, with a 3-dimensional lineality space. Its f vector modulo the lineality space is $(35, 157, 159)$.
\end{example}

We now define the symplectic counterpart of the usual Dressian $\Dr(k,2n)$; the \emph{symplectic Dressian} which we denote by $\Sp\Dr(k,2n)$:

\begin{definition}
The \emph{symplectic Dressian} $\Sp\Dr(k,2n)\supseteq \T\grass(k,2n)$ is the prevariety corresponding to the ideal $S_{k,2n}(\KK)$, i.e., the set of valuated matroids (tropical Pl\"ucker vectors) $\mu$ satisfying:
\begin{equation}
\label{eq:tropsymp}
\forall\,\, S\in \binom{[2n]}{k-2}, \text{ the minimum in }\quad \bigoplus_i \mu_{Si\overline{i}}\quad \text{is achieved at least twice.}    
\end{equation}

We call the elements of $\Sp\Dr(k,2n)$, \emph{tropical symplectic Pl\"ucker vectors}.
\end{definition}

We have seen the notion of a tropical linear space $L_{\mu}$ associated to a valuated matroid $\mu\in \Dr(k,2n)$ in Section \ref{sec:tropicalgrassmannian}. The following definition is a tropical analogue of being isotropic, which was first introduced by Rinc\'on in \cite{Fel12}. Under the scope of Coxeter matroids (Borovik et al \cite{BGW03}), Rinc\'on \cite{Fel12} studies valuated matroids corresponding to type {\tt D}, while here we are considering type {\tt C}. However, orthogonality and isotropicity in type {\tt D} vs type {\tt C} differ by just a sign, which becomes irrelevant under tropicalization, so his definition remains the relevant one for us:

\begin{definition}
We say that two points $x,y\in  \TT^{2n}$ are \emph{(tropically) orthogonal} if the minimum is achieved at least twice in
	\[
	\bigoplus_{i\in[2n]} x_i \odot y_{\overline{i}}.
	\]
A tropical linear space $L\subseteq \TT^{2n}$ is \emph{isotropic} if all of its points are pairwise \emph{orthogonal}.
\end{definition}

Recall the tropical Stiefel map $\pi$ from Definition \ref{def:stiefel}. In the following definition, we establish two more analogies for transversal valuated matroids:

\begin{definition}
Let $\mu\in \Dr(k,2n)$ and suppose $\mu =\pi(A|B)$, for $A,B\in \TT^{k\times n}$.
\begin{enumerate}
\item If the rows of $(A|B)$ are (tropically) orthogonal, we say $\mu$ has a \emph{row-orthogonal} presentation.
\item If $A\odot B^T$ is symmetric, we say $\mu$ has a \emph{symmetric} presentation. 
\end{enumerate}
\end{definition}

To summarize, we have the following tropical analogues of the five equivalent characterizations of isotropic linear subspaces with respect to a symplectic form:

\begin{enumerate}
	\item $\mu\in \T\grass_p(k,2n)$.
	\item $\mu\in \Sp\Dr(k,2n)$.
	\item $L_{\mu}$ is isotropic.
	\item $\mu$ has a row-orthogonal presentation.
	\item $\mu$ has a symmetric presentation.
\end{enumerate}

Since it turns out they do not agree in general, it is convenient to give a name to each of the sets they define, respectively, the:

\begin{enumerate}
	\item \emph{tropical symplectic Grassmannian.}
	\item \emph{symplectic Dressian.}
	\item \emph{isotropic region.}
	\item \emph{row-orthogonal Stiefel image.}
	\item \emph{symmetric Stiefel image.}
\end{enumerate}

\subsection{Implications between the Five Tropical Analogues} 
We are now set to examine when does one of the analogues above imply another.

To begin with, we examine the cases when $\Sp \Dr(k,2n)$ is equal to $\T\grass(k,2n)$. It turns out that the two fans are equal when $k=2$, as stated in the following theorem. We keep the proof for Section \ref{sec:ranktwo}, where we study the rank two case in detail.

\begin{theorem}\label{thm:rank2}
The set of Pl\"ucker and symplectic relations forms a tropical basis for the ideal $S_{2,2n}(\KK)$, i.e., the symplectic Dressian $\Sp\Dr(2,2n)$ and the tropical symplectic Grassmannian $\T\grass_p(2,2n)$ coincide. In particular, $\T\grass_p(2,2n)$ does not depend on $p$.
\end{theorem}

In \Cref{exa:zoo-one} of the matroid zoo (\Cref{sec:matroidzoo}), we describe an explicit point that is in $\Sp\Dr(3,6)$ but not in $\T\grass(3,6)$, hence showing that these two fans are not equal.
This is in contrast with type ${\tt A}$, where $\Dr(3,6) = \T\Gr(3,6)$.
With the help of this example, we also find the missing relation to obtain a tropical basis for $\grass_2(3,6)$. The  computational information of the fan $\Sp\Dr(3,6)$ is shown below:

\begin{example}
The symplectic Dressian $\Sp \Dr(3,6)$ is a non-pure non-simplicial 7-dimensional polyhedral  fan in $\TP^{19}$ with f vector $\{35, 153, 155, 2\}$. It has 149 maximal cones; of these, 2 are of dimension 7 while 147 are of dimension 6. However, the rays of $\Sp \Dr(3,6)$ are the rays of $\T\grass(3,6)$.
\end{example}
We will show in \Cref{thm:main}, that the non-equality of the two fans for $k=n=3$ discussed above holds for all $k\geq 3, n\geq 3$.\\

We now go on to study connections between the tropical symplectic Grassmannian, the symplectic Dressian and isotropic tropical linear spaces. We first show that a tropical linear space corresponding to any valuated matroid in $\T\grass(k,n)$ is isotropic:

\begin{proposition}
\label{prop:tropgriso}
Let $\mu\in \T\grass_p(k,2n)$ be a valuated matroid in the tropical symplectic Grassmannian. Then $L_\mu$ is isotropic. 
\end{proposition}
\begin{proof}
Consider $\KK$ an algebraically closed field of characteristic $p$ such that $\val: \KK \to \TT$ is surjective, for example the field of formal Hahn series (see for example \cite{JS21}). Since $\mu \in \T\grass_p(k,2n)$, then there exists an isotropic linear space $L\subseteq \KK^{2n}$ realizing $\mu$, that is, its Pl\"ucker vector $\X$ satisfies $\val(\X) = \mu$ (here we are using that $\KK$ has a surjective valuation). Since $L$ is given by \Cref{eq:linspace} and $L_\mu$ is given by \Cref{eq:troplin}, then by \cite[Theorem 3.2.5 and Proposition 4.1.6]{MS15}, we have that $L_\mu = \val(L)$. Since any two points $x,y\in L$ satisfy $\omega(x,y) = 0$, we have that $\val(x)$ and $\val(y)$ are orthogonal, so $L_\mu$ is isotropic. 
\end{proof}

The following is a result due to Rinc\'on \cite{Fel12}. It gives a sufficient condition for a tropical linear space to be isotropic in the Lagrangian case.

\begin{proposition}[\cite{Fel12}]
\label{prop:isorincon}
Let $\mu\in \Dr(n,2n)$ be a valuated matroid. Then $L_{\mu}$ is isotropic if and only if $\mu_{\overline{\Jm}}= \mu_{[2n] \setminus \Jm}$ for all $\Jm\in \binom{[2n]}{n}$.
\end{proposition}

Notice that for $k=n\le 3$, the relations $\mu_{\overline{\Jm}}= \mu_{[2n] \setminus \Jm}$ are equivalent to the symplectic relations (in other words, these relations are trivially true), so we have the following result:

\begin{proposition}
\label{prop:dressiso}
For $k=n\le 3$, $L_{\mu}$ is isotropic if and only if $\mu\in \Sp\Dr(n,2n)$.
\end{proposition}

\begin{remark}
The equations from \Cref{prop:isorincon} are also known to cut out the Lagrangian Grassmannian in the non-tropical case for any $n$. That is, a Pl\"ucker vector $\X$ satisfies $\X_{\overline{\Jm}}= \X_{[2n] \setminus \Jm}$ if and only if $\X \in \Sp\Gr(n,2n)$ \cite[Theorem 5.16]{KAR18}.
\end{remark}

The proof of \Cref{prop:isorincon} is based on the observation that the dual linear space $L_{\mu^*}$ (where $\mu^*_\Jm := \mu_{[2n]\setminus \Jm}$) consists of all points that are orthogonal to $L_\mu$ with respect to the tropical dot product. So being isotropic means that $L_{\overline{\mu}} \subseteq L_{\mu^*}$, where $\overline{\mu}_{\Jm} := \mu_{\overline{\Jm}}$. For the Lagrangian case, $\dim(L_{\overline{\mu}}) = \dim(L_{\mu^*})= n$ so they must be the same, hence \Cref{prop:isorincon}. However, for $k<n$, we can combine this approach with the \Cref{prop:incidence} to obtain a characterization of isotropic tropical linear spaces in terms of the Pl\"ucker coordinates, although not as clean as \Cref{prop:isorincon}.

\begin{proposition}
    A tropical linear space $L_{\mu}$ is isotropic if and only if 
\[
\forall\,\, S_1, S_2 \in\binom{[2n]}{k-1}, \text{ the minimum in } \bigoplus_i \mu_{S_1i}\odot \mu_{S_2\bar{i}}
\]
is achieved twice (if $i\in S_1$ or $\bar{i} \in S_2$ we assume $\mu_{S_1i} = \infty$ or $\mu_{S_2\bar{i}} = \infty$ respectively).
\end{proposition}

Finally, we show that the symmetric Stiefel image is contained in the row-orthogonal Stiefel image:
\begin{proposition}\label{prop:orthosym}
Let $A, B\in \TT^{k\times n}$ such that $A\odot B^T$ is symmetric. Then the rows of $(A|B)$ are (tropically) orthogonal.
\end{proposition}

\begin{proof}
Since $A\odot B^T$ is a $k\times k$ symmetric matrix, we have that for any $i, j$: 
\[\bigoplus_{l\in[n]} A_{il} \odot B_{jl} = \bigoplus_{l\in[n]} A_{j l} \odot B_{il}.\]

This implies that rows $i$ and $j$ of $(A|B)$ are orthogonal, since
\[
\bigoplus_{l\in[2n]} (A|B)_{il} \odot (A|B)_{j\bar{l}} = \left(\bigoplus_{l\in[n]} A_{il}  \odot B_{j l} \right)\oplus \left(\bigoplus_{l\in[n]} B_{il}  \odot A_{j l}\right).
\]

\end{proof}

The following is our main theorem. It combines the results from this section and Sections \ref{sec:ranktwo}, \ref{sec:directsums}, \ref{sec:matroidzoo} to give a full description of all the implications.
\begin{theorem}
\label{thm:main}
None of the statements (1)-(5) are equivalent in general. The following is a complete list of the implications between them for each $k$ and $n$:
\begin{itemize}
	\item[(a)] In general, $(1)\Rightarrow (2)$ and $(1)\Rightarrow (3)$.
	\item[(b)] For transversal valuated matroids, $(3)\Rightarrow (4)$ and $(5)\Rightarrow (4)$.
	\item[(c)] For $k=1$, they are all equivalent.
	\item[(d)] For $k=2$, $(1) \Leftrightarrow (2)$.
	\item[(e)] For $k=n\le 3$, $(2) \Leftrightarrow (3)$.
\end{itemize}
There is an example showing that any other implication fails (even when restricting to transversal valuated matroids, for the cases involving (4) and (5)).
\end{theorem}

Notice that set $(1)$, the tropical symplectic Grassmannian, is special in that it depends on the characteristic $p$. However, \Cref{thm:main} remains true regardless of the characteristic chosen.

\begin{proof}
The proof is summarized in \Cref{fig:zoomap}, the map of the matroid zoo.

\begin{itemize}
	\item[(a)] $(1)\Rightarrow (2)$ is trivial and $(1)\Rightarrow (3)$ is \Cref{prop:tropgriso}. 
	\item[(b)] For transversal valuated matroids, $(3)\Rightarrow (4)$ is immediate, since the rows of $A$ are always points in $L_{\pi(A)}$ and  $(5)\Rightarrow (4)$ is \Cref{prop:orthosym}.
	\item[(c)] For $k=1$ all spaces are equal to the usual Dressian $\Dr(1,2n)$.
	\item[(d)] This is \Cref{thm:rank2}.
	\item[(e)] This is \Cref{prop:dressiso}.
\end{itemize}

Now to show that there are no other implications, we provide a complete list of \emph{minimal} counterexamples in \Cref{sec:matroidzoo}.
We call them minimal in the following sense: there is no example of rank $k'\le k$ on $2n'<2n$ elements with $n'-k'\le n -k$ showing that the same implication fails.
The reason for this notion of minimality is that we can use \Cref{lemma:sum} to extend any counterexample to a counterexample with more elements, without decreasing the rank nor the co-rank. So it is enough to consider the following minimal counterexamples:

\begin{itemize}
    \item \Cref{exa:zoo-one} shows that neither (2) nor (3) imply (1) for $k=n=3$.
    \item Examples \ref{ex:4points} and \ref{ex:cube} show that (3) does not imply (2) for $k=2$, $n=3$ and $k=n=4$, respectively. \Cref{ex:4points} is also a minimal example for (3) not implying (1).
    \item \Cref{ex:D} shows that none imply (4), for $k=n=2$.
    \item \Cref{ex:E} shows that neither (4) nor (5) imply (1), (2) or (3), for $k=n=2$.
    \item Examples \ref{ex:2lines} and \ref{ex:G} show that (2) does not imply (3) or (4), for $k=3$, $n=4$ and $k=n=4$, respectively.
\end{itemize}

\end{proof}

\section{Rank Two Case}\label{sec:ranktwo}
In this section, we prove Theorem \ref{thm:rank2} and deduce from it enumerative information for $\T\grass(2,2n)$.

\subsection{Coincidence of the Tropical Symplectic Grassmannian with the Symplectic Dressian}
Notice that the ideal $\I_{2,2n}$ for $\grass(2,2n)$ is generated by the following $\binom{2n}{4}$ Pl\"ucker relations
\begin{equation}\label{eqn:relationsrk2}
\X_{i,j}\X_{k,l}-\X_{i,k}\X_{j,l}+\X_{i,l}\X_{j,k}=0\quad \text{for}\quad 1\leq i<j<k<l\leq \ov{n},
\end{equation}
and a single linear relation
\begin{equation}\label{eq:rk2}
    \X_{1,\ov{1}} +\cdots+\X_{n,\ov{n}}=0.
\end{equation}

\subsection*{Proof of ~Theorem \ref{thm:rank2}}
We want to prove that $\Sp\Dr(2,2n)=\T\grass_p(2,2n)$.
We know that isotropic linear spaces satisfy a single linear relation \eqref{eq:rk2} when $k=2$.

Obviously, $\T\grass_p(2,2n) \subseteq \Sp\Dr(2,2n)$, so let $\mu \in \Sp\Dr(2,2n)$. We have that the minimum in $\mu_{1,\ov{1}}\oplus\dots\oplus \mu_{n,\ov{n}}$ is achieved twice. By Theorem \ref{thm:rank2A} we know that $\mu\in \T\Gr(2,2n)$. Let $\KK$ be an algebraically closed field of characteristic $p$ and $\KK\ldb t \rdb$ be the field of Puiseux series with coefficients in $\KK$. Suppose $\mu$ has rational entries. Then there exists a linear subspace $L \subseteq (\KK\ldb t \rdb)^{2n}$ whose Pl\"ucker vector $\X$ satisfies $\val(\X_I) = \mu_I$ for every $I\in \binom{[2n]}{2}$. 

It may be that $L$ is not isotropic, i.e., $\X$ does not satisfy \Cref{eq:rk2}. However, we have that the minimum of $$\bigoplus\limits_{i\in [n]}\val(\X_{i,\ov{i}})$$ is achieved twice. Let $i$ be such that $\val(\X_{i,\ov{i}})$ achieves the minimum and assume it is not infinite (otherwise we would already have that $L$ is isotropic). Consider a matrix $A\subseteq \KK\ldb t \rdb^{2\times 2n}$ such that its rows form a basis of $L$.

Notice that multiplying column $i$ by a factor $\alpha\in \KK\ldb t \rdb$ with valuation $0$, changes $\X_I$ by a factor of $\alpha$ if $i\in I$, otherwise it remains unchanged. However, since $\val(\alpha) = 0$ we observe that $\val(\X_I)$ remains unchanged anyway. We are going to show that there exists such $\alpha$ such that the tropical linear space $L'$, generated by the rows of the matrix $A'$ which is the result of multiplying column $i$ of $A$ by $\alpha$, is isotropic. By the discussion above, the valuation of the Pl\"ucker vector of $L'$ is $\mu$ and this will show that $\mu$ is in the tropical symplectic Grassmannian.

Let $a = \val(\X_{1,\ov{1}}+\dots+ \X_{n,\ov{n}})$. If $a= \infty$ we have already that $L$ is isotropic. So suppose $a$ is finite. By the axioms of valuation, $a\geq \val(\X_{i,\ov{i}})$.
If $a=\val(\X_{i,\ov{i}})$, consider 
\[
b = \frac{\overline{-t^{-a}\sum \limits_{j\ne i}^n \X_{j,\ov{j}}}}{\overline{t^{-a}\X_{i,\ov{i}}}}
\]
where $\overline{x}$ denotes the value of $x$ in the residue field (which in this case is $\KK$). Since the minimum is achieved twice for $\val(\X_{j,\ov{j}})$, without loss of generality we can assume that $i$ is such that $b \ne 0$. We have now that 
\[
\val(\X_{1,\ov{1}} + \cdots + b\X_{i,\ov{i}}+\cdots+ \X_{n,\ov{n}}) >\val(\X_{i,\ov{i}}).
\]
Let $m$ be least common denominator of all the exponents of $t$ in all entries of $A$. We are going to construct $\alpha$ recursively as follows: 
\begin{enumerate}
    \item Start with $\alpha_1 = b$ where $b=1$ if $a>\val(\X_{i,\ov{i}})$ and let $S_1 = \X_{1,\ov{1}} + \cdots + \alpha_1\X_{i,\ov{i}}+\cdots + \X_{n,\ov{n}}$.
    \item Let $\lambda_j = \val(S_j) - \val(\X_{i,\ov{i}})$ which is a positive rational and let $c_j = \overline{-t^{-\val(S_j)} S_j}$.
    \item Define $\alpha_{j+1} = \alpha_j +t^{\lambda_j}c_j$  and $S_{j+1} = \X_{1,\ov{1}} + \cdots + \alpha_{j+1}\X_{i,\ov{i}}+\cdots +\X_{n,\ov{n}}$.
    \item Repeat the two previous steps infinitely (or until $S_j = 0$). 
    \item Let $\alpha$ be the resulting series, that is 
    \[\alpha := b + \sum_{j=1}^\infty c_j t^{\lambda_j}.\]
\end{enumerate}
We have that $\alpha$ is a Puiseux series since all $\lambda_j$ share the denominator $m$. Furthermore $\val(\alpha) = 0$. Notice that $\val(S_{j+1}) \ge \val(S_j) +\dfrac{1}{m}$, so we must have that $\val(\X_{1,\ov{1}} + \cdots+  \alpha \X_{i,\ov{i}}+\cdots +\X_{n,\ov{n}}) = \infty$ as we wanted. 
\qed

\subsection{Coincidence with the Space of Phylogenetic Trees}
We consider the space of phylogenetic trees with $2n$ labeled leaves (studied by Billera et al \cite{BHV01}) and its underlying simplicial complex which we denote by $\mathbf{T}_{2n}$; first introduced by Buneman \cite{BNM74} and later studied by Robinson and Whitehouse \cite{RW96} and Vogtmann \cite{VTN90}. 

We prove the following proposition which relates $\T\grass(2,2n)$ to the simplicial complex $\mathbf{T}_{2n}$. 

\begin{proposition}\label{pro:coincidence}
The tropical symplectic Grassmannian $\T\grass(2,2n)$ is isomorphic as a fan to $\RR^n\times H_{n-2} \times \mathbf{T}_{2n}$ where $H_{n-2}$ denotes a generic hyperplane in $\TT\PP^{n-1}$.
\end{proposition}

\begin{proof}
Let $Y$ denote the tropical hyperplane given by $\mu_{1,\ov{1}}\oplus\dots\oplus \mu_{n,\ov{n}}$, which is the tropicalization of the linear relation  \eqref{eq:rk2}. Recall that the lineality space $L$ of $\T\Gr(2,2n)$ is the largest linear space contained in every cone of $\T\Gr(2,2n)$. It is the $(2n-1)$-dimensional vector space spanned by the vectors: 
\[u_i = \sum\limits_{B\ni i} e_B\]
for every $i\in [2n]$ (the span of this vectors in $\RR^{\binom{2n}{2}}$ contains the line $\RR (1,\dots,1)$ which is collapsed in the tropical projective plane). 

\Cref{thm:rank2} tells us that $\T\grass(2,2n)$ is the intersection of $\T\Gr(2,2n)$ with the tropical hyperplane $Y$. The intersection $Y\cap L$ is a fan with a $n$-dimensional lineality space $L'$ spanned by the vectors:
\[u_i-u_{\overline{i}} \] 
Every other cone of $L\cap Y$ is of the form $L'+\cone\{u_i \mid i\in I\}$ for any subset $I\subseteq[n]$ of size at most $n-2$. So $L\cap Y$  is isomorphic as a fan to $\RR^n\times H_{n-2}$. 

For any point $\mu\in \TT\PP^{{\binom{2n}{2}}-1}$, it is straightforward to see that there is a point $\mu'\in \mu+L$ where the minimum in \Cref{eq:rk2} is achieved at every term. From this we can deduce that $(\mu+L)\cap Y = \mu'+(L\cap Y)$. 
One way to interpret the last equation is that $Y$ intersects $\T\Gr(2,2n)$ through its lineality space.
As a fan, $\T\Gr(2,2n)$ is isomorphic to $L\times \mathbf{T}_{2n}$ \cite[Theorem 3.4]{SS04}. So we conclude that 
\[\T\grass(2,2n) \cong (L\cap Y) \times \mathbf{T}_{2n}\cong \RR^n\times H_{n-2} \times \mathbf{T}_{2n}.\]
\end{proof}

We can now extract enumerative information of $\T\grass(2,2n)$:

\begin{corollary}\label{cor:rays}
The fan $\T\grass(2,2n)$ has $2^{2n-1}-n-1$ rays and $\binom{n}{2}[(4n-5)!!]$ facets.
\end{corollary}
\begin{proof}
The number of rays of $H_{n-2}$ is $n$ and the number of rays in $\mathbf{T}_{2n}$ is $2^{2n-1}-2n-1$ \cite{BHV01}, so the number of rays of $\T\grass(2,2n)$ is $n+ 2^{2n-1}-2n-1 = 2^{2n-1}-n-1$.

The number of facets of $H_{n-2}$  is $\binom{n}{2}$ and the number of facets in $\mathbf{T}_{2n}$ is $(4n-5)!! = 1\times3\times5\times\cdots\times(4n-5)$ \cite{RW96, VTN90}, so the number of facets of $\T\grass(2,2n)$ is $\binom{n}{2}[(4n-5)!!]$.
\end{proof}
 
One can turn any simplicial fan into a simplicial complex by quotienting out the lineality space and intersecting with the sphere of radius 1. We say that the homology of the simplicial fan is the simplicial homology of this complex. The following corollary specifies some topological information of $\T\grass(2,2n)$; its homotopy type in particular:

\begin{corollary}\label{cor:betti}
The Betti numbers for the simplicial complex corresponding to the fan $\T\grass(2,2n)$ are as follows: for $n\ge 4$ the only positive Betti numbers are $b_0 =1$, $b_{n-3} = n-1$,  $b_{2n-4} =
(2n-2)!$ and $b_{3n-7} = (n-1)(2n-2)!$; for $n=3$ they are $b_0 =3$ and $b_2 = 72$ and for $n=2$ we only have $b_0 =3$.
\end{corollary}
\begin{proof}
The simplicial complex $\mathbf{T}_{2n}$ has the homotopy type of a bouquet of $(2n-2)!$ spheres of dimension $2n-4$ \cite{RW96, VTN90} while the complex corresponding to $H_{n-2}$ has the homotopy type of a bouquet of $n-1$ spheres of dimension $n-3$. Now we use K\"unneth formula for product of spaces \cite{Kun23} to compute the Betti numbers of the simplicial complex of $\T\grass(2,2n)$ as stated in the corollary.
\end{proof}

\section{Direct Sums}\label{sec:directsums}
Let $\mu_1\in \Dr(k_1,n_1)$ and $\mu_2 \in \Dr(k_2, n_2)$ be two valuated matroids. The \emph{direct sum} between them  $\mu_1\oplus \mu_2 \in \Dr(k_1+k_2,n_1+n_2)$ is defined as
\[
(\mu_1\oplus \mu_2)_{\Jm_1\sqcup \Jm_2} := {\mu_1}_{\Jm_1}+{\mu_2}_{\Jm_2}
\]
where $\Jm_1\subset [n_1]$ and $\Jm_2\subset [n_2]$. Notice that this implies that $(\mu_1\oplus \mu_2)_{\Jm} = \infty$ if $|\Jm\cap n_1| \ne k_1$.

The main motivation behind this construction is that if $\mu_1$ and $\mu_2$ are representable valuated matroids, represented by $L_1$ and $L_2$, then $\mu_1\oplus \mu_2$ is representable and represented by $L_1\times L_2$. It is straightforward to see that 
\begin{equation}
L_{\mu_1\oplus \mu_2} = L_{\mu_1}\times L_{\mu_2}
\label{eq:bergman_sum}
\end{equation}
and that if $\mu_1= \pi(A_1)$ and $\mu_2 = \pi(A_2)$ then 
\begin{equation}
\mu_1\oplus \mu_2 = \begin{pmatrix}
\pi(A_1) & \infty \\
\infty & \pi(A_2)	
\end{pmatrix}.
\label{eq:stiefel_sum}
\end{equation}
Moreover, the converse is also true: every presentation of $\mu_1\oplus \mu_2$ is of this form. This follows from \Cref{eq:bergman_sum} and \cite[Theorem 6.6]{FO19}.

From the observations above we can prove the following statements which help us to extend the counter-examples from \Cref{sec:matroidzoo} to counter-examples for all suitable $k$ and $n$.
\begin{lemma}
\label{lemma:sum}
Let $X$ be any of the five sets considered in \Cref{thm:main} for all $k$ and $n$ and let $\mu$ be a valuated matroid. Then the following are equivalent:
\begin{enumerate}
	\item $\mu\in X$,
	\item $\mu\oplus U_{0,2}\in X$,
	\item $\mu\oplus U_{1,2} \in X$,
\end{enumerate}
where $U_{0,2} = (0) \in \T\Gr(0,2)$ and $U_{1,2} = (0,0) \in \T\Gr(1,2)$ are the uniform matroids.
\end{lemma}
\begin{proof}
We are going to examine each set that we are considering one by one:

\begin{itemize}
	\item \textbf{The tropical symplectic Grassmannian:} If $\mu\in \T\grass(k,2n)$ then there is an isotropic linear space $L\subset \KK^{2n}$ such that $\mu=\val(L)$. Then the space $L' = L\times \{(0,0)\} \subset \KK^{2n+2}$ is obviously isotropic (when pairing the last two coordinates) and represents $L\oplus U_{0,2}$. Any representation of $L\oplus U_{0,2}$ is of this form so the other direction holds. 
	
	Similarly, $L'' = L\times \{(x,x) \mid x\in \KK\}$ is a $(k+1)$-dimensional vector space representing $\mu\oplus U_{1,2}$ which is isotropic if and only if $L$ is isotropic. 
	
	\item \textbf{The symplectic Dressian:} Since $(\mu\oplus U_{0,2})_{\Jm} =\infty$ if $\Jm\cap\{2n+1,2n+2\} \ne \emptyset$, the only non-trivial symplectic relations occur when $S$ in \Cref{eq:tropsymp} does not contain $2n+1$ or $2n+2$ and they are equivalent to the symplectic relations on $\mu$. 
	
	For $\mu\oplus U_{1,2}$, notice that $(\mu\oplus U_{1,2})_{\Jm} =\infty$ unless $|\Jm\cap\{2n+1,2n+2\}|=1$. So again we have that the only non-trivial symplectic relation occurs when $|S\cap\{2n+1,2n+2\}| = 1$ and they are again equivalent to those for $\mu$ (since $(\mu\oplus U_{1,2})_{\Jm} = \mu_{\Jm\setminus k}$ for $k\in \{2n+1,2n+2\}$.
	\item \textbf{The isotropic region:} By \Cref{eq:bergman_sum}, $L_{\mu\oplus U_{0,2}} = L_{\mu}\times\{(\infty,\infty)\}$ and $L_{\mu\oplus U_{1,2}} = L_{\mu}\times \{(x,x) \mid x\in \TT \}$, which are clearly isotropic if and only if $L_{\mu}$ is isotropic.
	
	\item \textbf{Row-orthogonal Stiefel image:} Suppose $\mu$ is a transversal valuated matroid. We have that $\mu=\pi(A)$ for a tropical matrix $A\in \KK^{k\times 2n}$ if and only if the matrix $A'$ which is  obtained by adding two columns to $A$ with all entries equal to $\infty$ satisfies $\pi(A') = \mu\oplus U_{0,2}$. Moreover, all presentations of  $\mu\oplus U_{0,2}$ are of this form. 
	
	Similarly, consider the matrix $A''$ that consists of adding a row to $A'$ whose first $2n$ entries are equal to $\infty$ and the last two are equal to $0$. We have that $\pi(A) = \mu \Leftrightarrow \pi(A'') = \mu\oplus U_{1,2}$ and all presentations are of this form. 
	
	Now it is clear that the rows of $A$ are orthogonal if and only if the rows of $A'$ are orthogonal and if and only if the rows of $A''$ are orthogonal. 
	
	\item \textbf{Symmetric Stiefel image:} Let us write a presentation of $\mu$ as $(A| B)$ and the corresponding presentations of $\mu\oplus U_{0,2}$ and $\mu \oplus U_{1,2}$ as $(A'|B')$ and $(A''| B'')$, so for example 
	\[
	A' = \begin{pmatrix}
		A & \infty
	\end{pmatrix} \in \TT^{k\times (n+1)}
	\]
and
	\[
	A'' = \begin{pmatrix}
		A & \infty\\
		\infty & 0
	\end{pmatrix} \in \TT^{(k+1)\times (n+1)}.
	\]
Then $A' \odot B'^T = A\odot B$ and
\[
A''\odot B''^T = \begin{pmatrix}
		A\odot B^T  & \infty\\
		\infty & 0
	\end{pmatrix} \in \TT^{(k+1)\times (k+1)},
\]
from which the desired statement follows. 
\end{itemize}
\end{proof}

\section{The Matroid Zoo}\label{sec:matroidzoo}
Welcome to the matroid zoo, a section devoted to show examples that exhibit different pathologies.
Examples \ref{exa:zoo-one} through \ref{ex:G} are used in the proof of \Cref{thm:main}. Additional highlights of the zoo include:
\begin{itemize}
    \item \Cref{exa:zoo-one} shows that the tropical symplectic Grassmannian depends on the characteristic.
    \item \Cref{ex:4points} shows that neither being in the tropical symplectic Grassmannian nor in the Dressian is closed under subspaces (as opposed to the classical symplectic Grassmannian).
    \item  \Cref{exa:nonrepresentableflag} shows a non realizable flag with symplectically realizable components.
    \item  \Cref{exa:zoo-nine} shows that (non-realizable) symplectic matroids do not always satisfy the tropical symplectic relations.
\end{itemize}

\subsection{In the Symplectic Dressian but not in the Tropical Symplectic Grassmannian}
\begin{examplemat}\label{exa:zoo-one}
Let $M_{K_4}$ be the matroid whose affine representation is depicted in \Cref{fig:k4}. 
\begin{figure}[h]
	\centering
		\includegraphics[width= 0.25\textwidth]{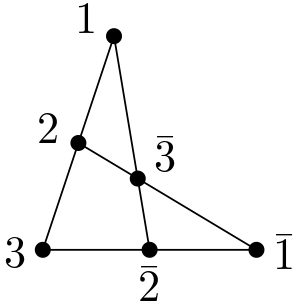}
	\caption{The matroid $M_{K_4}$}
	\label{fig:k4}
\end{figure}
We call it $M_{K_4}$ because it is the graphical matroid that corresponds to the complete graph $K_4$ with edges $i$ and $\overline{i}$ opposite to each other. It is easy to see that this matroid satisfies all the symplectic relations, since all of its non-bases are admissible sets. However, notice that the polynomial
\[\X_{123}\X_{\bar{3}\bar{2}\bar{1}}-\X_{12\bar{3}}\X_{3\bar{2}\bar{1}}-\X_{13\bar{2}}\X_{2\bar{3}\bar{1}}+\X_{1\bar{3}\bar{2}}\X_{23\bar{1}} + 2\X_{12\bar{2}}\X_{3\bar{3}\bar{1}}
\]
is in the ideal $S_{3,6}(\KK)$. However, unless the characteristic is 2, the tropicalization of the polynomial attains the minimum only in $x_{12\bar{2}}\odot x_{3\bar{3}\bar{1}}$,
 so $M_{K_4} \notin \T\grass_p(3,6)$ for $p\ne 2$. However, $M_{K_4} \in \T\grass_p(3,6)$, so already with 6 elements the tropical symplectic Grassmannian depends on the characteristic. This is remarkable, since for type ${\tt A}$ the smallest non realizable matroid has 7 elements (the Fano and non-Fano matroids) and for type ${\tt D}$ it is known that the smallest non-realizable tropical Wick vector has at least 12 elements and at most 14 elements.
\end{examplemat}

We checked computationally that for $\grass_2(3,6)$, the Pl\"ucker relations, the symplectic relations and the relation 
\[
\X_{123}\X_{\bar{3}\bar{2}\bar{1}}-\X_{12\bar{3}}\X_{3\bar{2}\bar{1}}-\X_{13\bar{2}}\X_{2\bar{3}\bar{1}}+\X_{1\bar{3}\bar{2}}\X_{23\bar{1}}
\]
form a tropical basis. The relation above is interesting as it consists of only admissible sets. Perhaps this can shine a light into the relation to symplectic matroids (see \Cref{subsec:sympmat}).
\subsection{Isotropic but not in the Symplectic Dressian.}
\begin{examplemat}
\label{ex:4points}
Consider the matroid $M$ with $2n= 6$ elements and rank $2$ with parallelism classes given by $1\bar{1}$, $2\bar{2}$, $3$ and $\bar{3}$ (see \Cref{fig:4points}). This matroid clearly does not satisfy the only symplectic relation $x_{1\bar{1}}\oplus x_{2\bar{2}}\oplus x_{3\bar{3}}$. However, the Bergman fan of $L_M$ consists of four rays: $\cone(e_{1\bar{1}})$, $\cone(e_{2\bar{2}})$, $\cone(e_3)$ and $\cone(e_{\bar{3}})$. It is easy to see that any two points on these 4 rays are orthogonal to each other, since they all have at least 4 coordinates achieving the minimum, so this space is isotropic.

\begin{figure}[h]
	\centering
		\includegraphics[width= 0.38\textwidth]{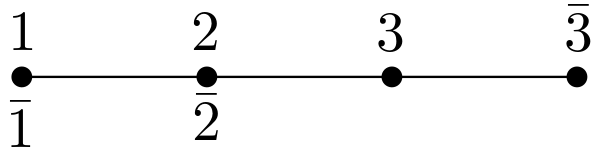}
	\caption{The matroid in \Cref{ex:4points}}
	\label{fig:4points}
\end{figure}

Moreover, notice that $L_M$ is a subspace of the Bergman fan of the uniform matroid $U_{3,6}$, which does not lie in the tropical symplectic Grassmannian. So this example also shows that being a subspace of a tropical linear space in the tropical symplectic Grassmannian (or symplectic Dressian) does not guarantee being in tropical symplectic Grassmannian (respectively symplectic Dressian).
This is somewhat anti-intuitive, since any subspace of a linear space in the classical symplectic Grassmannian is obviously also in the symplectic Grassmannian (of the corresponding rank). 

One could think this would extend over to the tropical case, by taking any isotropic linear space $H$ that realizes $U_{3,6}$ and intersecting it with the pre-image under $\val$ of $L_M$. However, any linear subspace of rank two of $H$ that has points mapped to $\cone(e_{1\bar{1}})$ and $\cone(e_{2\bar{2}})$ via $\val$, must also have points mapped to $\cone(e_{3\bar{3}})$ (this follows from $H$ being isotropic). So the tropicalization of any such subspace would have to be $L_{U_{1,2}\oplus U_{1,2} \oplus U_{1,2}}$. 
\end{examplemat}

The previous example shows that if $k<n$, then there are tropical isotropic linear spaces that are not in the symplectic Dressian.
However, for $k=n\le 3$ a tropical linear space is isotropic if and only if it is in the Dressian. 
The remaining cases are $k=n\ge4$. 
The following are the tropical isotropic relations that are not tropical symplectic relations for $k=n=4$:
\begin{enumerate}
	\item $x_{12\bar{1}\bar{2}}\oplus x_{34\bar{3}\bar{4}}$,
	\item $x_{13\bar{1}\bar{3}}\oplus x_{24\bar{2}\bar{4}}$,
	\item $x_{14\bar{1}\bar{4}}\oplus x_{23\bar{2}\bar{3}}$,
\end{enumerate}
while the following are the  tropical symplectic relations which are not isotropic:
\begin{enumerate}
	\item $x_{12\bar{1}\bar{2}}\oplus x_{13\bar{1}\bar{3}}\oplus x_{14\bar{1}\bar{4}}$,
	\item $x_{12\bar{1}\bar{2}}\oplus x_{23\bar{2}\bar{3}}\oplus x_{24\bar{2}\bar{4}}$,
	\item $x_{13\bar{1}\bar{3}}\oplus x_{23\bar{2}\bar{3}}\oplus x_{34\bar{3}\bar{4}}$,
	\item $x_{14\bar{1}\bar{4}}\oplus x_{24\bar{2}\bar{4}}\oplus x_{34\bar{3}\bar{4}}$.
\end{enumerate}
\begin{examplemat}
\label{ex:cube}
Consider the matrix
\[
A = \begin{pmatrix}
1&1&1&1&1&1&1&1\\
0&1&1&0&0&1&1&0\\
0&0&2&2&0&0&2&2\\
0&0&0&0&1&1&1&1	
\end{pmatrix},
\]
and let $M$ be the matroid given by the columns of $A$ as vectors (see \Cref{fig:trapezoidal_prism}). We have that $13\bar{1}\bar{3}$ and $24\bar{2}\bar{4}$ are bases while $12\bar{1}\bar{2}, 23\bar{2}\bar{3}, 34\bar{3}\bar{4}$ and $14\bar{1}\bar{4}$ are not. So $M$ does not satisfy any of the four symplectic relations listed above, but it does satisfy all isotropic relations, (to see that it satisfies the relations that are both isotropic and symplectic, notice all sets of the form $ijk\overline{i}$ are bases). 
\begin{figure}[h]
	\centering
		\includegraphics[width= 0.34\textwidth]{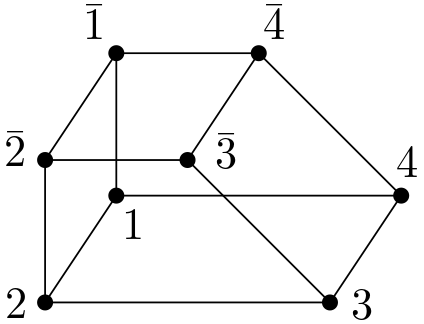}
	\caption{Affine representation of \Cref{ex:cube}.}
	\label{fig:trapezoidal_prism}
\end{figure}
 
\end{examplemat}

\subsection{With Row-orthogonal Presentation but without Symmetric Presentation}
\begin{examplemat}\label{ex:D}
Consider the the matrix $(A|B) = \left(\begin{smallmatrix}
0&0&1&1\\
0&0&0&0
\end{smallmatrix}\right)$ and the valuated matroid $\mu = \pi(A|B)$. The rows of $(A|B)$ are orthogonal, however, $A\odot B^T = \left(\begin{smallmatrix}
1&0\\
1&0
\end{smallmatrix}\right)$ is not symmetric. Moreover, every presentation $(A'|B')$ of $\mu$ is the result of replacing one of the 1's in $(A|B)$ by some $x\ge 1$ and replacing either $A_{2,1}$ or $A_{2,2}$ by some $y\ge 0$ \cite[Example 3.3]{FO19}. So in any case $A'\odot B'^T = \left(\begin{smallmatrix}
1&0\\
1&0
\end{smallmatrix}\right)$.
\end{examplemat}

\subsection{With Symmetric Presentation but not Isotropic}

\begin{examplemat}\label{ex:E}
Consider the the matrix $(A|B) = \left(\begin{smallmatrix}
0&1&0&1\\
0&0&0&0
\end{smallmatrix}\right)$ and the valuated matroid $\mu = \pi(A|B)$. We have that $A\odot B^T = \left(\begin{smallmatrix}
0&0\\
0&0
\end{smallmatrix}\right)$ is symmetric. However, $\mu_{1\overline{1}} = 0$ and $\mu_{2\overline{2}} = 1$ so it does not satisfy the isotropic relations from \Cref{prop:isorincon}. 

\end{examplemat}

\subsection{In the Symplectic Dressian but without Row-orthogonal Presentation}
	
Notice that by \Cref{prop:dressiso}, the smallest example of a transversal valuated matroid in the symplectic Dressian that does not have a presentation with orthogonal rows must have at least 8 elements.
To completely deal with all other cases, we need an example of rank 3 and another of rank 4. 

\begin{examplemat}
\label{ex:2lines}
Consider the rank 3 matroid $M$ whose affine representation is given by \Cref{fig:2lines}. We have that $M$ is transversal (by \cite[Theorem 4.6]{BD72}, for example) and, by \cite[Proposition 4.3]{FO19}, any valuated presentation $A\in\pi^{-1}(M)$ of $M$ must have a row equal to $R_1 = (0,\infty,\infty,0,0,\infty,\infty,0)$ (modulo tropical scalar multiplication, i.e., adding a constant) and another row equal to $R_2 = (0,0,0,\infty,\infty,0,0,\infty)$, which are not orthogonal to each other. 
\begin{figure}[h]
	\centering
		\includegraphics[width= 0.38\textwidth]{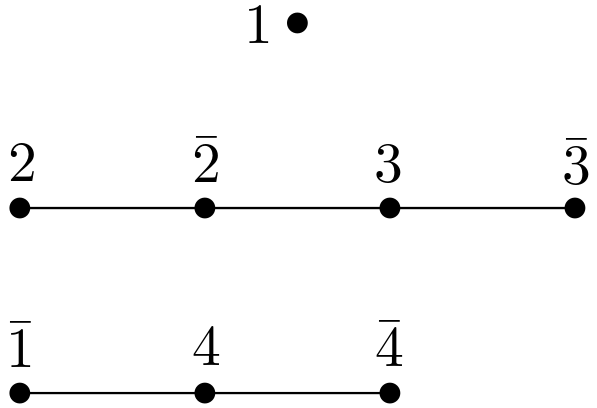}
	\caption{The matroid in \Cref{ex:2lines}}
	\label{fig:2lines}
\end{figure}
\end{examplemat}

\begin{examplemat}\label{ex:G}
Now consider the rank 4 matroid $M$, whose affine representation is almost as \Cref{fig:2lines}, but with $23\bar{2}\bar{3}$ being co-planar instead of co-linear. In other words, $M$ is the matroid whose only cyclic flats are $23\bar{2}\bar{3}$ of rank 3 and $4\bar{1}\bar{4}$ of rank two. Again, this matroid is transversal and, by \cite[Proposition 4.3]{FO19}, any valuated presentation $A\in\pi^{-1}(M)$ of $M$ must have a row equal to $R_1$. However, now there are two rows which should have $\infty$ entries at $4,\bar{1}$ and $\bar{4}$, but are not necessarily equal to $R_2$. One of them could be equal to $(\infty,0,0,\infty,\infty,0,0,\infty)$ and be orthogonal with $R_1$, but not both at the same time (other wise it forces $1$ to be in the $4\bar{1}\bar{4}$ line). So there is always a row which is not orthogonal to $R_1$.
\end{examplemat}

\subsection{Non-realizable Flag with Symplectically Realizable Components}
\begin{examplemat}\label{exa:nonrepresentableflag}
Consider the non-Pappus matroid $P$ depicted in \Cref{fig:nonpappus}. 

\begin{figure}[h]
	\centering
		\includegraphics[width= 0.34\textwidth]{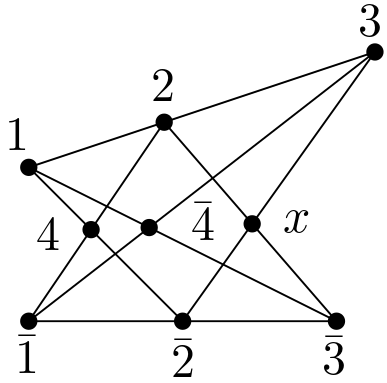}
	\caption{The non-Pappus matroid.}
	\label{fig:nonpappus}
\end{figure}

Let $M_1= P\backslash x$ and $M_2 = P/x$. 
The non realizability of the non-Pappus matroid results in the non-realizability of the flag $\Bergman(M_1) \subseteq  \Bergman(M_2)$. This is Example 7 from \cite{BGW03}, but we have added a labeling i.e., a pairing of the points so that $M_1$ and $M_2$ are symplectically realizable. 
We checked computationally that $M_2 \in \T\grass_0(3,8)$ while $M_1\in \T\grass_0(2,8)$ by \Cref{thm:rank2}. 
\end{examplemat}

\subsection{Symplectic Matroid not in the Symplectic Dressian}
\begin{examplemat}

Consider the symplectic matroid depicted in \Cref{fig:nonrepsymp} (the definition of symplectic matroid can be found below in \Cref{sec:future}). 
Explicitly, we have the symplectic matroid with non bases $\{12,1\bar{2},\bar{1}3,\bar{1}\bar{3}\}$. This example appeared in \cite[\S 3.4.3]{BGW03} as an example of a non-realizable symplectic matroid.

\begin{figure}[h]
	\centering
		\includegraphics[width=0.38\textwidth]{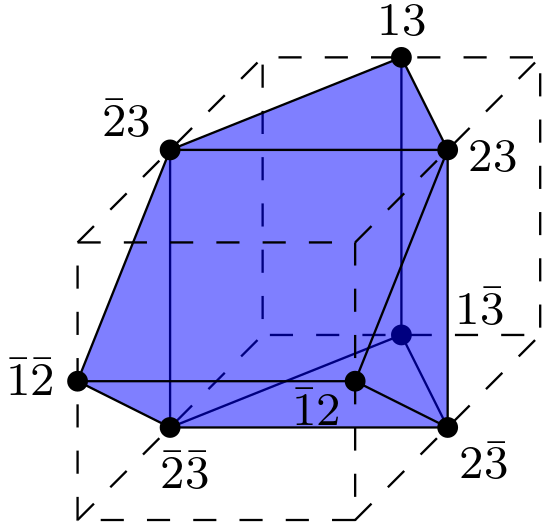}
	\label{fig:nonrepsymp}
	\caption{The polytope of a non-realizable symplectic matroid that does not satisfy the symplectic relations.}
\end{figure}

If we were to complete this symplectic matroid to a usual matroid (i.e., determining its bases for non-admissible sets), we would necessarily have that $13\bar{3}$ and $\bar{1}2\bar{2}$ are the two parallel classes. But then the symplectic relation $x_{1\bar{1}}\oplus x_{2\bar{2}}\oplus x_{3\bar{3}}$ achieves its minimum once, so such matroid is not in the symplectic Dressian.

\label{exa:zoo-nine}

\end{examplemat}

\section{Conormal Bundles}
\label{sec:conormal}

Recall that the dual of a valuated matroid $\mu$ on $[n]$ of rank $k$ is a valuated matroid $\mu^*$ on $[n]$ of rank $n-k$ where $\mu^*_B = \mu_{[n]\backslash B}$. 

In \cite{ADH20}, Ardila, Denham and Huh studied the \emph{conormal fan} of a matroid, $\Sigma_{M,M^*}$, whose support equals $L_{M\oplus M^*}$. The latter is a tropical linear space of dimension $n$ on $[2n]$. The conormal fan serves as a tropical analogue of conic Lagrangian subvarieties. Therefore, a natural question is: where does this kind of tropical linear space fit in our picture (that is, in \Cref{fig:zoomap})? 

We answer this question for a more general class of tropical linear spaces, namely for the following natural generalization to valuated matroids:

\begin{definition}
Let $\mu$ be a valuated matroid. The \emph{conormal bundle} of $\mu$ is the tropical linear space $L_{\mu \oplus \mu^*}$.
\end{definition}

We remark that this is a generalization of conormal fans only in the same way as tropical linear spaces generalize Bergman fans: for unvaluated matroids the conormal bundle coincides with the support of the conormal fan, but the conormal fan may have a finer subdivision. As the name suggests, for valuated matroids the conormal bundle may not be a fan. 

\begin{theorem}
\label{prop:conormal}
Let $\mu$ be a valuated matroid. Then:
\begin{enumerate}
    \item if $\mu$ is realizable over a field $\KK$ of characteristic $p$, then $\mu\oplus \mu^*\in \T\grass_p(n,2n)$.
    \item for any $\mu$, $\mu\oplus\mu^*\in \Sp\Dr(n,2n)$, i.e., $\mu\oplus\mu^*$ satisfies the tropical symplectic relations.
    \item the conormal bundle $L_{\mu\oplus\mu^*}$ is isotropic.
    \item if $\mu$ and $\mu^*$ are transversal, $\mu\oplus\mu^*$ has a symmetric presentation if and only if $\mu$ has a single basis.
\end{enumerate}
\end{theorem}

\begin{proof}
For the first statement, if $\mu=\trop(V)$ for a linear subspace $V\subseteq \KK^n$, then the orthogonal complement $V^\perp$ satisfies that $\trop(V^\perp)= \mu^*$ and $V\times V^\perp\subseteq \KK^{2n}$ is isotropic with respect to the canonical symplectic form given by \Cref{eq:sympform}. So $\mu\oplus\mu^*= \trop(V\times V^\perp)\in \T\grass_p(n,2n)$. 

Next, let $\mu$ be any valuated matroid (not necessarily realizable). Notice that the tropical symplectic relations for $\mu\oplus \mu^*$ are exactly the same as the tropical Pl\"ucker relations for $\mu$:

\[
    \bigoplus_{i}(\mu\oplus\mu^*)_{Si\bar{i}}
    = \bigoplus_{i}\mu_{(S\cap [n])\cup i}\odot \mu^*_{S\backslash [n] \cup \bar{i}} = \bigoplus_{i}\mu_{S'i}\odot \mu_{T\backslash i}
\]
where $S'= S\cap [n]$ and $T = [n]\backslash \bar{S}$. 

For the third statement, we use \Cref{prop:isorincon}. Let $J= B_1\sqcup B_2\in \binom{2n}{n}$ where $B_1 = J \cap [n]$. Then
\[
    (\mu\oplus\mu^*)_{[2n] \backslash J}
    = \mu_{[n]\backslash B_1}\odot \mu^*_{[n]\backslash B_2}
    = \mu_{B_2} \odot \mu^*_{B_1}
    = (\mu\oplus\mu^*)_{\bar{J}}.
\]

Finally, to relate conormal bundles with the Stiefel image, we need to work under the assumption that both $\mu$ and $\mu^*$ are transversal (transversality is closed under direct sums but not under duality). If such is the case, then by part (2) we have that $\mu\oplus\mu^*$ has row-orthogonal presentation. 

Now, if $\mu$ has a single basis, it means it is the direct sum of loops and coloops (an element $i$ is a loop of $\mu$ if it is not in any basis and it is a coloop if it is a loop of $\mu^*$). Therefore, $\mu^*$ and $\mu\oplus\mu^*$ are also of this form. Then, if column $i$ of a presentation $(A|B)$ of $\mu\oplus\mu^*$ has a finite entry, the column $\bar{i}$ must have all entries equal to $\infty$.
From this it follows that $A\odot B^T$ has all infinity entries for a presentation $(A|B)$ and in particular it is symmetric.

If $\mu$ is not of this form, then there is an element $x\in [n]$ which is not a loop in $\mu$ or in $\mu^*$. Therefore, there are rows $i$ and $j$ in $(A|B)$ with a finite entry at $x$ and $\bar{x}$ respectively. This implies that the $(i,j)$-entry of $A\odot B^T$ is finite. However, the $(j,i)$ entry of $A\odot B^T$ is given by the tropical dot product of two all $\infty$ vectors, so it is equal to $\infty$. Therefore $A\odot B^T$ is not symmetric.
\end{proof}

\section{Future Bridges}
\label{sec:future}
In this section, we discuss several notions including the tropical symplectic flag variety, symplectic flag Dressian, flag Stiefel image, degenerations of symplectic flag varieties and total positivity and pose several questions about them.

\subsection{Symplectic Matroids} \label{subsec:sympmat} We follow Borovik et al \cite{BGW98,BGW03} and we refer the reader there for details on symplectic matroids.\\

We begin by recalling that we call a subset $\Jm\subset[2n]$ \emph{admissible} if for every $i\in \Jm$, $\bar{i}\notin \Jm$.
Let $e_{\bar{i}} := -e_i$ and for an admissible set $\Jm\subseteq [2n]$ consider $e_{\Jm} := \sum\limits_{i\in \Jm} e_i \in \RR^{[n]}$. 

\begin{definition}
\label{def:sympmat}
A \emph{symplectic matroid} is a collection $\BB\subseteq \binom{[2n]}{k}$ of admissible sets such that every edge of the polytope $\conv\{e_{\Jm} \mid \Jm\in \BB\}$ is either a translation of $e_i-e_j$ for some $i,j$ or of $2e_i$ for some $i$. 
\end{definition}

The main motivation behind this definition is the following:

\begin{theorem}[\cite{BGW98}]
\label{thm:repsympmat}
Let $L\subseteq \KK^{[2n]}$ be an isotropic linear space. Then the collection of admissible bases of $L$ is a symplectic matroid.
\end{theorem}

A symplectic matroid constructed as above is called \emph{representable}. From \Cref{thm:repsympmat}, it follows that any realizable symplectic matroid can be completed to a matroid which is in the tropical symplectic Grassmannian (as opposed to \Cref{exa:zoo-nine}). \\

One possible direction would be to find axioms for which one can enhance a symplectic matroid with a valuation on the tropical numbers, which includes every vector in the tropical symplectic Grassmannian $\T\grass(k,2n)$ and such that symplectic matroids are exactly the valuations that take values in $\{0,\infty\}$. However, even the $\{0,\infty\}$ symplectic matroids do not always satisfy the tropical symplectic relations (see \Cref{sec:matroidzoo} above), so this would not be contained in the symplectic Dressian $\Sp\Dr(k,2n)$.

In general, the connection of symplectic matroids with the tropical symplectic Grassmannian remains largely not understood. One reason for this is that the symplectic relations (\Cref{eq:tropsymp}) are not in terms of variables corresponding to admissible sets. However, there are equations in the ideal generated by them which are purely in terms of variables corresponding to admissible sets (at least in characteristic 2, see \Cref{exa:zoo-one}). Finding more of these relations and studying their combinatorics might lead to a better understanding of this connection.

\subsection{Characterization in terms of Matroid Subdivisions}
It would be desirable to obtain a characterization of tropical symplectic Pl\"ucker vectors in terms of polyhedral subdivisions, similar to Speyer's theorem for type {\tt A} \cite[Proposition 2.2]{Spe08}, generalized by Rinc\'on for any matroid polytope in  type {\tt A} as well as type {\tt D} \cite[Theorem 5.14]{Fel12} and in \cite[Theorem A]{BEZ20} to tropical flags. 

However, such a characterization is, in principle, not possible,
since satisfying the tropical symplectic relations is not closed under translations. This implies that being a tropical symplectic Pl\"ucker vector can not be determined from the matroid subdivision it induces.
But one could still ask which matroid subdivisions are induced by vectors from $\Sp\Dr(k,2n)$.

\subsection{The Tropical Symplectic Flag Variety and Symplectic Flag Dressian} The complete symplectic flag variety which will be denoted by $\complete_{2n}$, is the set of all full flags $\{\U_1\subset\cdots\subset \U_n, \,\, \dim \,\, \U_k =k\}$, with the extra condition that each $\U_k$ is isotropic with respect to the symplectic form $\omega$. We consider the Pl\"ucker embedding of $\complete_{2n}$ into the product of projective spaces $\mathbb{P}\Big(\bigwedge^k \mathbb{C}^{2n}\Big)$, for $1\leq k\leq n$. It is an irreducible projective variety of dimension $n^2$. With respect to this embedding, its defining ideal which we denote by $S_{2n}$, is seen to be generated by the corresponding Pl\"ucker and symplectic relations following the work of De Concini \cite{Dec79}. See also \cite{MAK20} and \cite{B20} for a further description and some examples of this ideal. Let $G_{2n}$ denote the set of the generators of the ideal $S_{2n}$.\\

Let $\textsc{TF}_{n}$ denote the usual tropical flag variety, i.e., the tropicalization of the (type {\tt A}$_{n-1}$) complete flag variety with respect to the Pl\"ucker embedding. Let $N=\sum_{i=1}^n \binom{2n}{i}$ denote the number of Pl\"ucker coordinates on $\complete_{2n}$ with respect to the Pl\"ucker embedding as discussed above. We introduce the symplectic counterpart of $\textsc{TF}_{n}$; the \emph{tropical symplectic flag variety}, which will be denoted by $\T\complete_{2n}$.

\begin{definition}
The \emph{tropical symplectic flag variety} $\T\complete_{2n}$, is the tropicalization of the symplectic flag variety $\complete_{2n}$. That is
\[\T\complete_{2n}=
\{\mu \in \TP^{N-1} \mid \mu \in \trop(V(f)) \quad \forall \,\, f \in S_{2n}\}.
\]
\end{definition}
 

\begin{example}[n=2]
The tropical symplectic flag variety $\T\complete_4$ is a pure 4-dimensional fan in $\mathbb{R}^8$ with a 2-dimensional lineality space. It has 10 rays and 15 maximal cones.
\end{example}

In the following, we define the \emph{symplectic flag Dressian} which we denote by $\complete\Dr_{2n}$. It is the analogue of the usual flag Dressian; introduced and studied by Brandt, Eur and Zhang \cite{BEZ20}.
\begin{definition}
The \emph{symplectic flag Dressian} $\complete\Dr_{2n}$ is given by:
\[\complete\Dr_{2n}=
\{\mu \in \TP^{N-1} \mid \mu \in \trop(V(f)) \quad \forall \,\, f \in G_{2n}\}.
\]
\end{definition}
As it would be expected, the tropical symplectic flag variety $\T\complete_{2n}$ is not equal to $\complete\Dr_{2n}$ in general. However, the two fans coincide computationally for $n=2$. It is important to note that the issue of representability of symplectic flags is subtle, in that there are flags with symplectically representable components but the flags themselves are not symplectically representable. Such a scenario is illustrated in Example \ref{exa:nonrepresentableflag} of the matroid zoo.\\

As in the classical setting, we can also obtain a flag of tropical linear spaces out of an $n\times n$ matrix, by considering all minors (not necessarily maximal) whose rows are indexed by $[k]$ for some $k\le n$.

\begin{definition}\label{def:flagstiefel}
The \emph{tropical flag Stiefel map} for type ${\tt A}$,
\[
\pi_{\mathcal{F}}^{\tt A}: \TT^{n\times n} \dashrightarrow \textsc{TF}_{n},
\] 
sends a matrix to the vector where for every non-empty proper subset $S\subset [n]$,  $\pi_{\mathcal{F}}(A)_S$ is the tropical determinant of the submatrix of $A$ with columns indexed by $S$ and the first $|S|$ rows. 

Similarly, we have a \emph{tropical flag Stiefel map} for type ${\tt C}$,
\[
\pi_{\mathcal{F}}^{\tt C}: \TT^{n\times 2n} \dashrightarrow \T\complete_{2n},
\] 
where we consider the minors indexed by non-empty subsets $S\subset [2n]$ of size at most $n$.

The image under the above maps we call \emph{flag Stiefel image}.
\end{definition}

\subsection{Degenerations of Symplectic Grassmann/Flag Varieties and Tropical Geometry} 
A cone of a tropical variety is said to be prime if the corresponding initial ideal is a prime ideal.
Prime maximal cones of tropicalizations of algebraic varieties give rise to a nice class of toric degenerations, namely, those that are irreducible. According to \cite{SS04}, all maximal cones of $\T\Gr(2,n)$ are prime, but not all those of $\T\Gr(3,6)$ are. On the other hand, Bossinger, Lambogila, Mincheva, and Mohammadi in \cite{BLMM17}, computed toric degenerations of the type {\tt A} flag variety corresponding to both prime and non-prime maximal cones in the tropicalization. So, it is natural to ask which of the maximal cones of the tropical symplectic Grassmann or flag variety are prime cones. 

Now, an explicit computation of tropicalizations of flag varieties is very cumbersome, in that, it is only possible for small $n$ (see again \cite{BLMM17}). However, a full facet description of a maximal prime cone with nice features in the type {\tt A} tropical flag variety has been given in \cite{FFFM19}, and connections made to various degenerations of flag varieties. For example, some of its facets correspond to linear degenerations \cite{GFFFR19}.

We point out that such a description can be explained through the tropical flag Stiefel map discussed above. Consider the cone $\mathcal{A} \subseteq \TT^{n\times n}$ of upper triangular matrices $A$ such that:

\begin{itemize}
    \item $A_{i,i} = 0$,
    \item $A_{i,j} = \infty,\qquad\qquad\qquad\qquad\qquad\quad$ for $j<i$,
    \item $A_{i,j} + A_{i+1,j+1} \ge A_{i+1,j}+A_{i,j+1},\quad\,$ for $1\leq i<j\leq n-2$.
\end{itemize}

The original description of the cone in \cite{FFFM19} was stated by setting the tropical Pl\"ucker coordinates as 
\[
\mu_B = \sum_{i=1}^l A_{p_i,q_i}
\]
where $p_1>\dots> p_l$ are the elements of $B \setminus [k]$, $q_1<\dots< q_l$ are the elements of $[k]\setminus B$ and $k=|B|$. However, we notice that this corresponds precisely to the minimal term in the tropical determinant of the submatrix with rows $[k]$ and columns $B$ got by adding the 0-terms $A_{i,i}$ for $i\in [k]\cap B$. To see that, notice that this is the lexicographic maximal permutation using finite coordinates and the third condition ensures that this is the minimal term. So we conclude that  the maximal prime cone described in \cite{FFFM19} is actually $\pi_{\mathcal{F}}^{\tt A}(\mathcal{A})$, where $\pi_{\mathcal{F}}^{\tt A}$ is the type {\tt A} tropical flag Stiefel map (\Cref{def:flagstiefel}). We see therefore that the valuated matroids within this cone are transversal.

A symplectic counterpart of this cone is under construction in \cite{B21} and we conjecture that it also lives in the tropical flag Stiefel image for type {\tt C}.

\subsection{Total Positivity}
A major object of study has been the non-negative Grassmannian $\Gr_{\ge0}(k,n)$, that is, the set of all Pl\"ucker vectors in $\Gr_\RR(k,n)$ with non-negative entries. Extensive research has been conducted on its positroid decomposition, specially motivated by its connections to quantum physics, such as computing Feynman integrals in the supersymmetric-Yang-Mills model (see, for example, \cite{ABT} for an overview). The cells of the positroid decomposition are indexed by a variety of combinatorial gadgets and can be parametrized using networks through the boundary measurement map \cite{POS06}. Adaptations of the combinatorics of the non-negative Grassmannian have been constructed both for type {\tt C}, more precisely, for the non-negative Lagrangian Grassmannian $\Sp\Gr_{\ge 0}(n,2n) := \Gr_{\ge 0}(n,2n)\cap \grass(n,2n)$ \cite{KAR18}. On the other hand, the study of the totally positive tropical Grassmannian as defined in \cite{SW05} has gained a lot of momentum recently (see for example \cite{ALS20,LPW20,SW20}) where the combinatorics of positroids play again a protagonist role. Therefore we believe it is likely that the combinatorial objects used in \cite{KAR18}, such as symmetric plabic graphs, can be used in a tropical setting to provide information about the non-negative part of the tropical symplectic Grassmannian.


\end{document}